\documentclass[twoside,11pt]{amsart} 
\usepackage{enumitem}
\usepackage{epsfig}
\usepackage{latexsym} 
\usepackage{amsmath} \usepackage{amssymb}
\usepackage{placeins}   % for FloatBarrier command



\newtheorem{theorem}{Theorem}[section]
\newtheorem{lemma}[theorem]{Lemma}

\newtheorem{conjecture}[theorem]{Conjecture}

\newdimen\epsfxsize
\newdimen\epsfysize

\newcommand {\gap}     {\makebox[0.075 in]{}}   
\newcommand {\biggap}     {\makebox[0.2 in]{}}   
\newcommand {\st}      {\gap : \gap}   
\newcommand{\lil}   {\scriptstyle }

\newcommand {\fto}     {\longrightarrow}

\newcommand {\set}[1]  {\left\{ {#1} \right\}}   
  
\newcommand {\pml}[1]  {{#1}^{\#}}

\newcommand{\Z}     {{\mathbb Z}}

\newcommand{\pgap}   {{\mathcal G}}

\begin{document}

\title[Surviving Eratosthenes sieve I:  quadratic density]{Surviving Eratosthenes sieve I: \\  quadratic density and Legendre's conjecture}

\date{21 Mar 2026}

%\institution{U of Washington}
\author{Fred B. Holt}
\address{fbholt62@gmail.com; www.primegaps.info; github.com/fbholt/Primegaps-v2}

\begin{abstract}
We have been studying Eratosthenes sieve as a discrete dynamic system, obtaining exact models for the relative populations
for small gaps (currently gaps $g \le 82$) in the cycle of gaps $\pgap(p^\#)$ at each stage of the sieve.
The gaps in the interval $\Delta H(p_k)=[p_k^2, p_{k+1}^2]$ are fixed in $\pgap(p^\#)$ and survive all subsequent stages of t
he sieve to be confirmed as gaps between primes.  

We have shown that samples of gaps between primes over these intervals of survival $\Delta H(p_k)$ have population distributions
that reflect the relative population models $w_g(p_k^\#)$.

This paper advances our study of the estimates of survival across stages of the sieve.  Inspired by Legendre's conjecture, 
we introduce the concept of quadratic density $\eta_s(p_k)$,
which is the expected population of the constellation $s$ in the intervals $[n^2, (n+1)^2]$ for $p_k \le n < p_{k+1}$.
We show that once a gap occurs in $\pgap(p^\#)$, its expected quadratic density increases across all subsequent stages of
the sieve.

Regarding Legendre's conjecture, beyond postulating one prime in the interval $[n^2,(n+1)^2]$, the quadratic density predicts 
the populations of several prime gaps within this interval.
\end{abstract}

\maketitle

\section*{Introduction}
By studying Eratosthenes sieve as a discrete dynamic system, we have previously \cite{FBHSFU, FBHPatterns} identified exact models of the
populations of small gaps and constellations within the cycles of gaps $\pgap(\pml{p})$ across all stages of the sieve.  
Currently we have these models $w_g(p_k^\#)$ for all gaps $g \le 82$ and $w_{g,J}(p_k^\#)$ for any admissible constellations of span 
$|s| \le 62$.

Here we explore in more detail how the populations of gaps and constellations in $\pgap(\pml{p})$
are reflected in the counts of gaps and constellations among prime numbers.  

Our working hypothesis is that the gaps between primes in the interval $\Delta H(p_k)=[p_k^2, p_{k+1}^2]$ is an 
approximately uniform sample from the
relative populations $w_g(p_k^\#)$ in the cycle of gaps $\pgap(p_k^\#)$.  

\begin{figure}[hbt]
\centering
\includegraphics[width=5.25in]{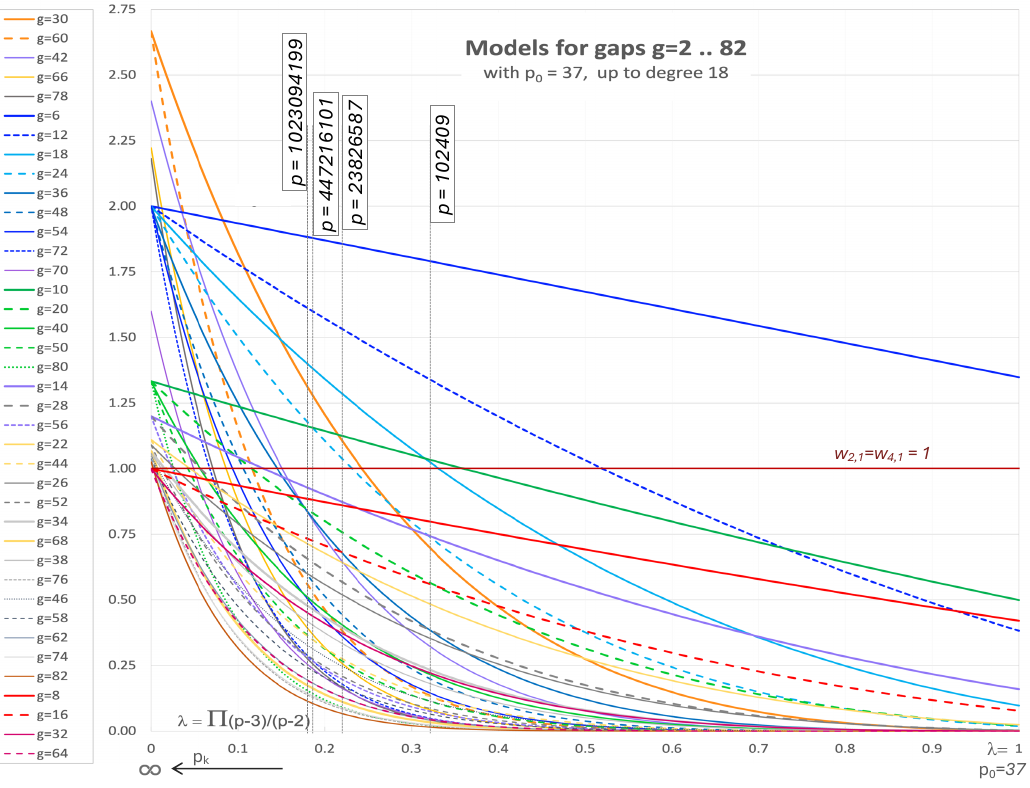}
\caption{\label{AllGapsFig} We have previously established the exact models for the populations $n_s(\pml{p_k})$ and relative populations $w_s(\pml{p_k})$
in the cycle of gaps $\pgap(\pml{p_k})$, for all constellations $s$ with span $|s| \le 2p_1$.  
Shown here are the relative population models $w_{g,1}(\pml{p_k})$ for $p_0=37$ and all gaps $g \le 82$.
The graphs start at the right, where $p_0=37$ and the system parameter $\lambda=1$.  As the sieve continues, $\lambda \fto 0$
as $p \fto \infty$, and we follow the graphs to the left.  As the sieve proceeds through enormous primes,
the graphs converge toward their asymptotic values at $\lambda=0$. }
\end{figure}

Samples of the gaps between primes over these intervals $\Delta H(p_k)$ consistently reflect the distributions indicated by the
relative population models $w_g(p_k^\#)$, to first order.  A selection of these sampled $\Delta H(p_k)$ are shown in 
Figure~\ref{DelHFig}.
 
\begin{figure}[hbt]
\centering
\includegraphics[width=5in]{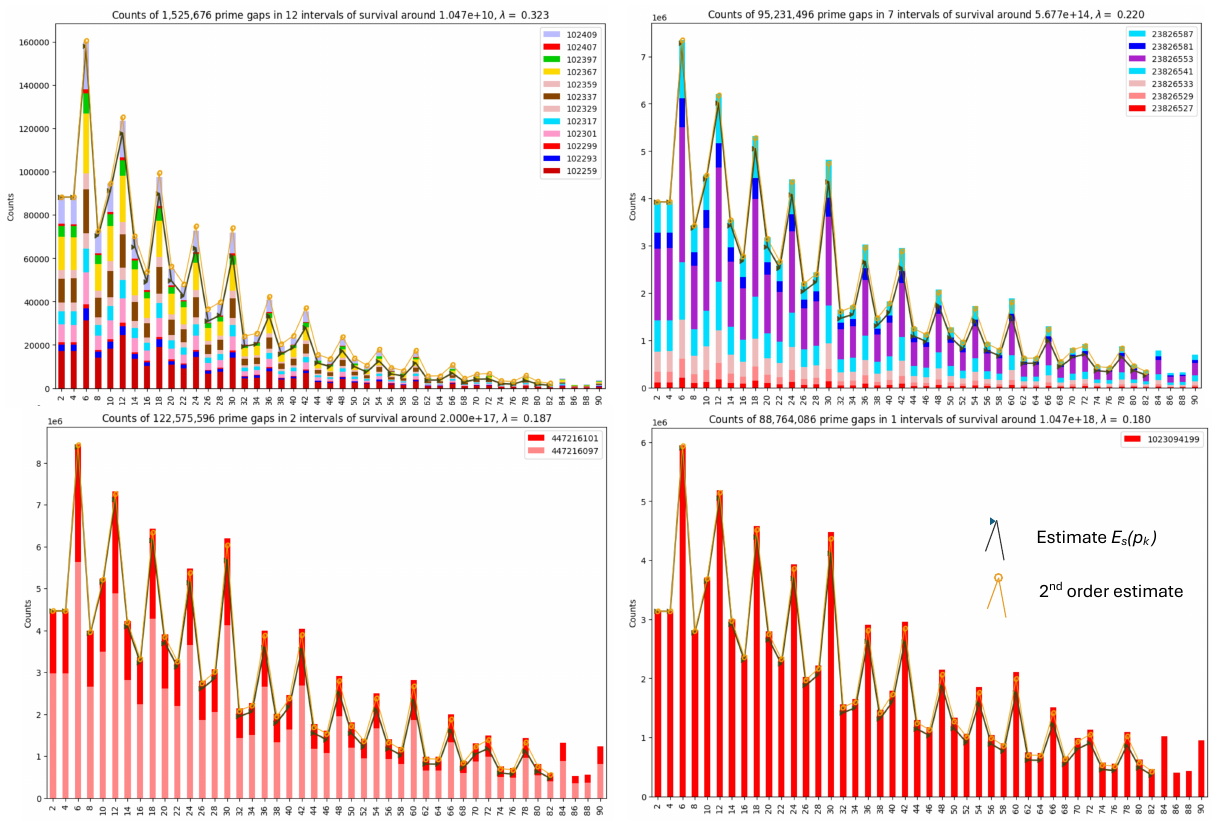}
\caption{\label{DelHFig} Shown here are samples of the populations of gaps within intervals of survival $\Delta H(p)$
for four different values of $\lambda$.  The first-order estimates based on $w_g(\lambda)$ and a $2^{\rm nd}$-order correction are
represented by the piecewise lines. The samples from the individual $\Delta H(p)$ are stacked and color-coded by
the gap ${g=p_{k+1}-p_k}$.}
\end{figure}

The span of $\Delta H(p_k)$ is proportional to the gap $g=p_{k+1}-p_k$.
$$|\Delta H(p_k)| = p_{k+1}^2-p_k^2=g(2 p_k+g)$$
and under our hypothesis of approximate uniformity we would expect to see this proportionality reflected in the populations of 
the gaps in this interval.
We do observe this proportionality across the bands within each column of figures like Figure~\ref{DelHFig}.

In this paper we make this make this measurement of proportionality more precise.
We define $\eta_s(p_k)$ to be the {\em quadratic density} of the constellation $s$ in the interval of survival $\Delta H(p_k)$.
$$ \eta_s(p_k) \; = \; \frac{1}{g} N_s([p_k^2, p_{k+1}^2]) \; \approx \; n_{s,J}(p_k^\#) \cdot \frac{2p_k+g}{p_k^\#}$$
The quadratic density $\eta_s(p_k)$ is the expected average occurrence of the constellation $s$ in the
$g$ quadratic intervals $[n^2, (n+1)^2]$ for $n=p_k, \ldots, p_{k+1}-1$, under the assumption of approximate uniformity
of the distribution of gaps in $\pgap(p_k^\#)$.

Figure~\ref{EtaSampleFig} shows the values of $\eta_g(p)$ around the first two samples in Figure~\ref{DelHFig}.
The first panel of Figure~\ref{EtaSampleFig} shows $\eta_g(p)$ for $506$ consecutive intervals of survival $\Delta H(p)$
for the primes $97897 \le p \le 103841$.  The distributions of the samples $\eta_g(p)$ are plotted, for the gaps $g$
for $2 \le g \le 94$, and we superimpose the mean population for this gap and two standard deviations both up and down.
The samples support our hypothesis of approximate uniformity.

Our study of quadratic densities $\eta_s$ of constellations $s$ across the intervals of survival $\Delta H(p)$ supports
Legendre's conjecture, which is one of Landau's problems from 1912.

\noindent {\bf Legendre's Conjecture (1808).} For every $n \ge 2$, there exists at least one prime number in the interval
$[n^2, (n+1)^2]$.

We show that for gaps, once a gap $g$ appears in $\pgap(p^\#)$, the expected value of its quadratic density $\eta_g$ increases
for all $q > p$.  For $n \ge 5$ not only do we expect one prime to occur in $[n^2,(n+1)^2]$, we have expected populations of
the prime gaps in these intervals, and we expect the populations of the gaps 
in the previous $\Delta H(p)$ to increase.  This is our Theorem~\ref{gapThm} below.

\begin{figure}[hbt]
\centering
\includegraphics[width=5in]{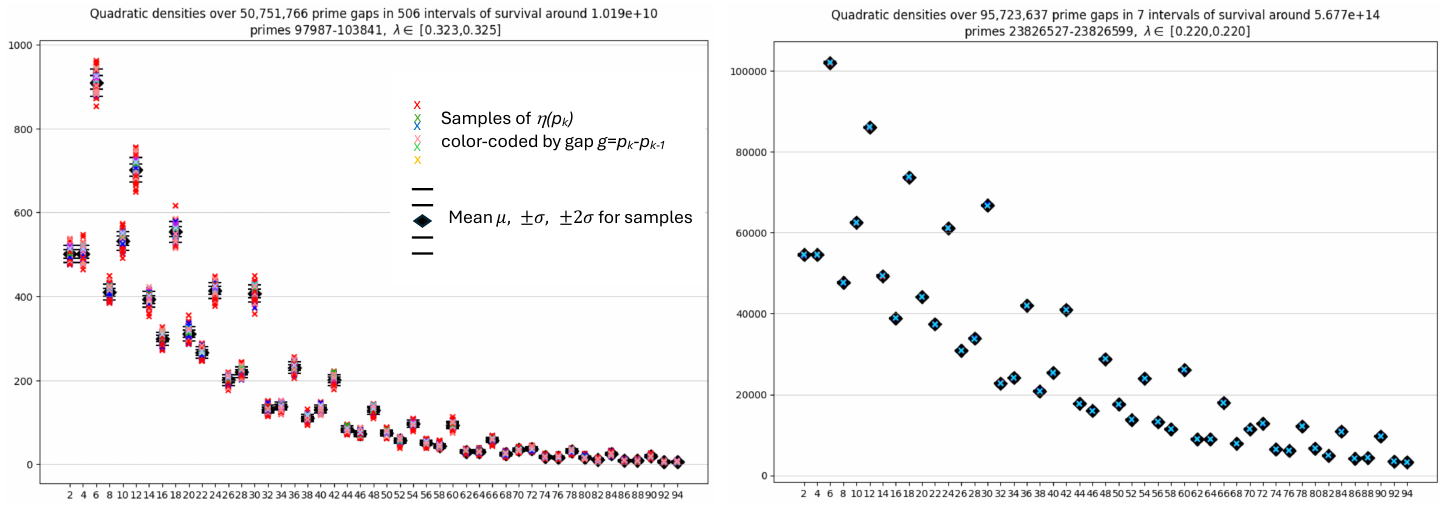}
\caption{\label{EtaSampleFig} These graphs are statistical summaries of samples of quadratic density $\eta$ corresponding to the
first two panels in Figure~\ref{DelHFig}. Shown are the samples of the quadratic densities themselves, as individual markers.  A black
diamond marks the mean for the sample, and the short lines mark one and two standard deviations of the sample.}
\end{figure}

These results add a constructive complement to Hardy \& Littlewood's \cite{HL} analytic estimates based on their circle method.
The prime number theorem provides a broad statistical confirmation that Legendre's Conjecture is true ``on average'' or true for almost
all $n$.  In 1923 Hardy \& Littlewood were able to produce estimates, e.g. their Conjecture~B, for the number of pairs of primes that 
differ by a given span $g$.  Brent \cite{BrentSmall} makes estimates for the gaps $g$ by applying inclusion-exclusion to Hardy
\& Littlewood's estimate.

Our approach breaks the occurrences of a span $g$ between pairs of $p$-rough numbers into the populations of the gap $g$ itself and its
driving terms of various lengths.  With exact population counts in the cycles $\pgap(p^\#)$, 
our hypothesis of approximate uniformity provides estimates for survival
that strongly align with the estimates by Hardy \& Littlewood and by Brent.  Rather than applying the circle method and then inclusion-exclusion, 
we have used the recursive construction of the cycles $\pgap(p^\#)$ to obtain exact population models for gaps and short admissible 
constellations, and we then estimate the number of occurrences in the interval of survival $\Delta H(p)$.

The quadratic density strikes a meaningful balance between the competing effects of the raw population of a gap growing by factors of
$(p-2)$ through the cycles of gaps $\pgap(p^\#)$, and the total number of gaps growing by $(p-1)$.  The population of a gap $g$ is growing
in the cycle of gaps $\pgap(p^\#)$ but its overall density is diminishing.

The interval of survival $\Delta H(p_k)$ provides manageable estimates, but samples using $\Delta H(p_k)$ directly have a lot of 
variation due to the factor $g_3=p_{k+1}-p_k$.  The quadratic density $\eta(p_k)$ neutralizes this factor, and its intuitive meaning is the average
population of gaps across the intervals $[n^2, (n+1)^2]$ for the $g_3$ such intervals in $\Delta H(p_k)$.

% XXXQHERE 26 Jan -- balance between two effects. population of gap g growing by factors of (p-2)
%  but total population of gaps is \phi(p#) which grows by factors of (p-1).  In the large every gap g goes to 0% of
% the total population.  In constant proportion to other gaps.
% The horizon of survival grows as g(2p+g).  The factor (2p+g) balances this tension but the factor g is erratic.

Jupyter notebooks of Python code and data files for this project are available at \textsf{https://github.com/fbholt/Primegaps-v2} .

% ========  SECTION:  setting, evolution =====================
\section{Evolution of populations of constellations in $\pgap(\pml{p})$.}
This section is a summary of results in \cite{FBHSFU, FBHPatterns, FBHktuple} about the exact models for 
the populations $n_s(\pml{p_k})$ and 
relative populations $w_s(\pml{p_k})$ of admissible constellations in the cycles of gaps $\pgap(\pml{p_k})$,
across stages of Eratosthenes sieve.  
In Section~\ref{EstSection} we use the relative population models to estimate
the distributions of gaps that we expect to survive the sieve.  In Section~\ref{etaSection} we examine the quadratic
density $\eta_s(p)$ of a constellation $s$ in $\Delta H(p)$, establishing some analytic results.

The data samples reflect the estimates to first order.  
That is, over the range that we have been able to compute, the gaps between the primes
 reflect the average expected survival of gaps, given the relative populations of gaps at the corresponding stage of Eratosthenes sieve.

In \cite{FBHSFU, FBHPatterns, FBHktuple} we study Eratosthenes sieve as a discrete dynamic system in order to study
the gaps and constellations among the prime numbers.

A sequence $s$ of consecutive gaps is a {\em constellation}.  The number of gaps in a constellation $s$ is its {\em length}, and
the sum of the gaps in $s$ is its {\em span}, denoted $|s|$.  A single gap is a constellation of length $J=1$.

At the stage of Eratosthenes sieve in which we have eliminated multiples of the prime $p$,
the remaining numbers are the $p$-rough numbers, the unit $1$ and those numbers all of whose divisors are greater than $p$.
The $p$-rough numbers consist of the generators of $\Z \bmod \pml{p}$ and their offsets by $\pml{p}$.
There is a cycle of gaps $\pgap(\pml{p})$ of length $\phi(\pml{p})$ and span $\pml{p}$ among the $p$-rough numbers.  

We identified \cite{FBH07} a three-step recursion that takes $\pgap(\pml{p_k})$ as input
and produces the next cycle $\pgap(\pml{p_{k+1}})$.  
Starting with $\pgap(\pml{p_k}) = g_1 g_2 \ldots g_{\phi(\pml{p_k})}$, the recursive construction is:
\begin{itemize}
\item[R1.]  identify the next prime $p_{k+1} = g_1 +1$.
\item[R2.]  Concatenate $p_{k+1}$ copies of $\pgap(\pml{p_k})$.
\item[R3.]  {\em Fusions.} Add the adjacent gaps ({\em fuse}) $g_1+g_2$ and then at the running sums indicated by the elementwise product
 $p_{k+1} \ast \pgap(\pml{p_k})$.
\end{itemize}
The result of this three-step recursion is the cycle $\pgap(\pml{p_{k+1}})$.

The first fusion in step R3, $g_1+g_2$, reflects the confirmation of $p_{k+1}$ as the next prime.
The subsequent $\phi(\pml{p_k})-1$ fusions correspond to removing the multiples of $p_{k+1}$.  When a candidate prime is
eliminated, we combine the gaps on either side of this former candidate.  These additions of adjacent gaps are the fusions in step R3.

The cycles of gaps under this recursion $\pgap(\pml{p_k}) \fto \pgap(\pml{p_{k+1}})$ form a discrete dynamic system.  In steps R2 and R3,
a lot of the structure of $\pgap(\pml{p_k})$ is preserved, and the fusions in R3 are applied methodically.  
$$ \pgap(\pml{p_0}) \fto \pgap(\pml{p_1}) \fto \cdots \fto \pgap(\pml{p_k}) \fto \pgap(\pml{p_{k+1}}) \fto \cdots $$

By the elementwise product in step R3 from $\pgap(\pml{p_k}) \fto \pgap(\pml{p_{k+1}})$,
we know that the minimum span between fusions is $2 p_{k+1}$.  So provided that
$|s| < 2p_{k+1}$, the possible fusions in $s$ all occur in separate images of $s$.  

If we take initial conditions from the cycle of gaps $\pgap(\pml{p_0})$, then we can derive exact models for the
populations $n_{s,j}(\pml{p})$ of driving terms of length $j$ and span $g=|s|$ in the cycle $\pgap(\pml{p})$ for all primes $p \ge p_0$
and all admissible constellations $s$ with span $|s| \le 2p_1$.  
See~\cite{FBHktuple}.
  
\begin{figure}[hbt]
\centering
\includegraphics[width=3.75in]{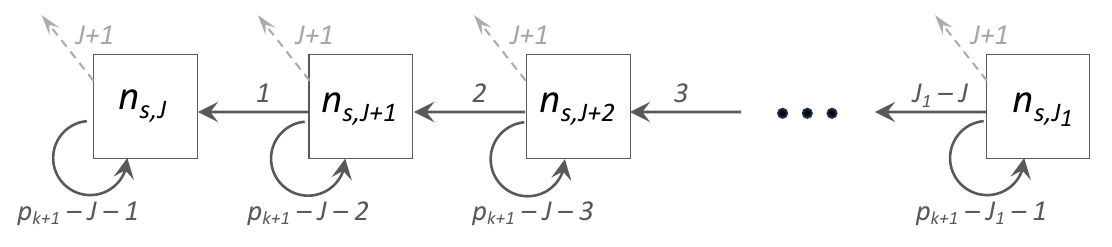}
\caption{\label{DynnFig} The model for the population of a constellation $s$ of length $J$ when the span $|s| < 2p_1$, 
illustrated as a Markov chain across its driving terms. }
\end{figure}

Figure~\ref{DynnFig} illustrates the population model.  For each driving term $\tilde{s}$ for $s$ of length $j$ in $\pgap(\pml{p_k})$,
of the $p_{k+1}$ images $\tilde{s}$ created in step R2, each of the $J+1$ boundary fusions eliminates an image as a driving term for $s$,
each of the $j-J$ interior fusions creates a driving term of length $j-1$, and remaining $p_{k+1}-j-1$ images survive intact.

For each length $J$ the populations $n_s(\pml{p_k})$ of every constellation $s$ of length $J$ and its driving terms are subject 
to the same transfer matrix
$$ n_s(\pml{p_{k+1}}) \; = \; \tilde{M}_{J:J_{\max}}(p_{k+1}) \cdot n_s(\pml{p_k}).$$
$n_{s,J}(\pml{p})$ denotes the population of the constellation $s$ itself in the cycle $\pgap(\pml{p})$.

The bound $2 p_1$ encourages us to use
as large a $p_0$ as we can manage.  In our work we use $p_0=37$ for gaps, so we can derive the exact population models for all gaps $g \le 82$.
Currently we use $p_0=29$ for admissible constellations, providing exact population models for any admissible constellations $s$ of span $|s|\le 62$.
New approaches to the enumerations \cite{Brown} of constellations within these large cycles may enable us to increase this starting point to 
$p_0 =41$ or $p_0 = 43$.

The populations $n_{s,j}(\pml{p})$ are all superexponential, dominated by the factor $\prod_{1}^{k}(p_i - j - 1)$.  To facilitate 
comparisons among the gaps, we derive the {\em relative population models} $w_{s,j}(\pml{p})$ 
$$ w_{s,j}(\pml{p_k}) \; = \; n_{g,j}(\pml{p_k}) / \prod_{p> J+1}^{p_k} (p - J-1) $$
where $J$ is the length of the constellation $s$.
At each stage of the sieve the relative population $w_{s,J}(\pml{p_k})$ represents the superexponential population $n_{s,J}(\pml{p_k})$
as a coefficient on $\prod_{p>J+1}^{p_k}(p - J-1)$.   These models are derived in \cite{FBHSFU, FBHPatterns, FBHktuple}.

\begin{figure}[hbt]
\centering
\includegraphics[width=3.75in]{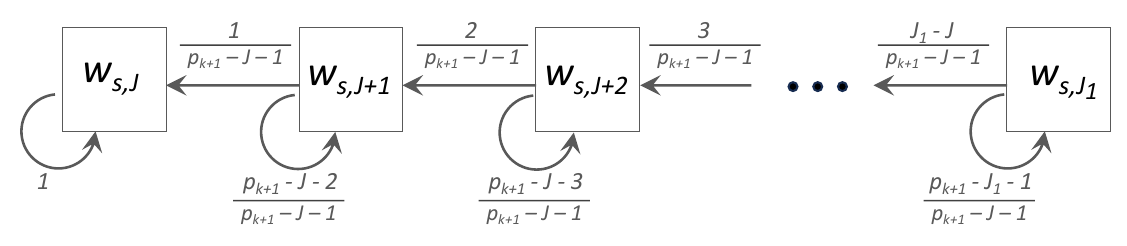}
\caption{\label{DynwFig} The model for the relative population $w_s(p_k^\#)$ 
of a constellation $s$ of length $J$ and span $|s| < 2p_1$, illustrated as a Markov chain. }
\end{figure}

The relative population models are linear dynamic systems 
$$ w_s(p_k^\#) \; = \; M^k w_s(p_0^\#).$$
The matrix $M^k$ is diagonalizable with spectrum
\begin{eqnarray*}
 \Lambda(M^k) & = & \left\{1, \; \lambda=\lambda_2, \;  \lambda_3, \ldots, \, \lambda_{J_1} \right\} \\
  {\rm where} \; \lambda_j  & = & \prod_{p_1}^{p_k} \frac{q-J-j}{q-J-1}
\end{eqnarray*}

We use the subdominant eigenvalue $\lambda = \lambda_2$ as the model parameter.  For the initial cycle $\pgap(p_0^\#)$
we have $\lambda(p_0)=1$, and as
$p_k \fto \infty$ the parameter $\lambda(p_k) \fto 0$.

% define the models w(lambda)
We have previously shown \cite{FBHSFU, FBHPatterns, FBHktuple} that the relative population
$w_{s,J}(\pml{p_k})$ of a constellation $s$ of length $J$ and span $|s|<2p_1$ is given by
\begin{eqnarray*}
w_{s,J} (\pml{p_k}) & = & w_{s,J}(\infty)  -  l_2 \cdot \lambda({\lil p_k})  +  l_3 \cdot  \lambda_3({\lil p_k})  - \cdots \\
 & & \biggap \cdots \biggap     + (-1)^{J_1-J}  l_{\lil J_1-J+1} \cdot \lambda_{\lil J_1-J+1}({\lil p_k}) \\
 & \approx & w_{s,J}(\infty) \; - \; l_2 \cdot \lambda \; + \; l_3 \cdot \lambda^2 \; - \cdots + \; (-1)^{J_1-J}  l_{J_1-J+1} \cdot \lambda^{J_1-J} 
\end{eqnarray*}
in which $J_1$ is the length of the longest admissible driving term for $s$,  
\begin{eqnarray*}
\lambda_j & = & \prod_{p_1}^{p_k}\frac{q-J-j}{q-J-1} \approx \lambda^{j-1} \\
 & {\rm and} & \\
l_j & = & L_j^T \cdot w_s(\pml{p_0})
\end{eqnarray*}
for left eigenvector $L_j^T$ and initial populations $w_{s,j}(p_0^\#)$.

The models for the relative populations $w_{g,1}(\pml{p_k})$ for gaps $g \le 82$ are displayed in Figure~\ref{AllGapsFig}.
These models start at the righthand side, with $p_0=37$ and the system parameter $\lambda=1$.  As the sieve continues, the models
evolve to the left, toward their asymptotic values at $\lambda=0$.

The convergence to the asymptotic values $w_{g,1}(\infty)$ is very slow.  We can use Merten's Third Theorem to estimate correspondences between
large values of $p_k$ and small values of $\lambda$.  To reduce $\lambda$ in half, say from $\lambda=c$ to $\lambda=\frac{c}{2}$, we need to 
square the value of the corresponding prime, from $p_k=q$ to $p_k \approx q^2$.

In $\pgap(\pml{37})$, as seen on the righthand side of Figure~\ref{AllGapsFig}, the relative populations of the gaps are ordered primarily by size.
The gap $g=6$, represented by the solid blue line, has already jumped out to be the most populous gap in $\pgap(\pml{p_k})$, and it will continue to 
be the most populous gap until $g=30$ passes it up, when $\lambda \approx 0.083$, which corresponds to $p_k \approx 3.386 E19$.

For any constellation $s$ of length $J$, its asymptotic value $w_{s,J}(\infty)$ is determined solely by the prime factors that divide a span between
boundary fusions in $s$.  Let $Q$ be the product of odd primes that divide a span between boundary fusions in $s$, then
\begin{equation}
w_{s,J}(\infty)  =  \prod_{p | Q} \frac{p-\nu_s(p)}{(p-J-1)_+} \label{EqWInf}
\end{equation}
Here $\nu_s(p)$ is the number of residue classes $\bmod p$ covered by an instance of $s$, and $(n)_+= {\rm max}\set{1, n}$.
For gaps, the length $J=1$ and this becomes
$$
 w_{g,1}(\infty)  =  \prod_{{\rm odd}\; q | g} \frac{q-1}{q-2}. 
$$
 
 These asymptotic values for the relative populations are consistent with the probabilistic predictions of Hardy \& Littlewood \cite{HL, BrentSmall}.
On the lefthand side of Figure~\ref{AllGapsFig} we see the gaps sort out into families that share the same odd prime factors $Q$.  
We have color-coded the curves by these families with the same $Q$.

These models are useful in the following ways.  
For every constellation $s$ we can calculate its asymptotic population $w_{s,J}(\infty)$, as given by
Equation~\ref{EqWInf}.
To determine the other coefficients $l_j$, we need the initial conditions (the populations of $g$ and all of its driving terms) for $g$ in some
cycle $\pgap(\pml{p_0})$ for which $g \le 2p_1$.  For gaps within this bound $g \le 2p_1$, 
we can determine the complete models for their relative populations, and
we can use these models to show us the evolution, as in Figure~\ref{AllGapsFig}, far beyond what we could ever enumerate directly.

To put the power of these models in perspective, the cycle $\pgap(\pml{61})$ with $1.54E22$ gaps lies at the horizon of current 
global computing capacity.
The cycle $\pgap(\pml{199})$ has more gaps than there are atoms in the observable universe.  Yet we will work below with results from
$\pgap(\pml{31259})$ and even longer cycles.

From the dynamics of the recursion we make the following conjecture.

\begin{conjecture} \label{UniConj} For large enough primes $p_k$, the populations of the gaps within $\pgap(\pml{p_k})$ are approximately 
uniformly distributed.
\end{conjecture}

Under this conjecture we estimate survival for the gaps in the cycle $\pgap(p^\#)$ as the gaps between primes in the intervals
$\Delta H(p_k) = [p_k^2, p_{k+1}^2]$.

\FloatBarrier

% ====== SECTION:  Models for surviving the sieve  ============================
\section{Estimating survival through the sieve}\label{EstSection}
In this section we develop and analyze an estimate $E_s(p_k)$ of how many instances of a constellation $s$ survive the stage of
Eratosthenes sieve for $p_k$, to be confirmed as constellations among prime numbers.

In $\pgap(\pml{p_k})$ the smallest composite number remaining is $p_{k+1}^2$.  Let ${\gamma_i = p_{k+1}^2}$ be the 
generator in $\pgap(\pml{p_k})$, 
and all the gaps from $g_2$ through $g_{i-1}$ will be confirmed as gaps between primes.  These gaps survive the sieve.  We call this range
up through $p_{k+1}^2$ the {\em horizon of survival} $H(p_k)$ for $\pgap(\pml{p_k})$.

However, in $H(p_k)$ the gaps within the previous horizon $H(p_{k-1})$ have already been confirmed as gaps between primes.  The interval 
$$ \Delta H(p_k) \; = H(p_k) \setminus H(p_{k-1}) \; = \; [p_k^2, \; p_{k+1}^2]$$
is the {\em interval of survival}, and this interval $\Delta H(p_k)$ is the unique interval of gaps in $\pgap(\pml{p_k})$ that 
are newly confirmed as gaps between primes at this stage of the sieve.

We would expect the relative populations in $\Delta H(p_k)$ to closely reflect the
relative populations $w(\pml{p_k})$.  Figure~\ref{DelHFig} shows the aggregated populations of gaps over some intervals of survival around
four different values of $\lambda$.
$$
\begin{array}{rrrrr}
\lambda = & 0.180 & 0.187 & 0.220 & 0.323 \\
{\rm for} \; p_k \sim & 1023094199 & 447216101 & 23826581 & 102409 
\end{array}
$$

To the extent that the distribution of constellations in the cycles $\pgap(\pml{p})$ is a uniform distribution, we expect the populations
of gaps in the interval of survival $\Delta H(p)$ to reflect the relative population models $w(\lambda)$ for the value of $\lambda$ corresponding to
$p$.  These would be the sections of $w(\lambda)$  in Figure~\ref{AllGapsFig}.  We observe that these expectations hold, to first order.

The gaps in $\pgap(\pml{p_k})$ are not random.  The broader conjecture above provides a foundation for our estimates of surviving 
Eratosthenes sieve.  For our estimates a narrower conjecture would suffice.

\begin{conjecture}\label{UniHConj}
For large enough primes, the populations of the gaps within $\Delta H(p_k)$ are approximately 
uniformly sampled from their populations in $\pgap(\pml{p_k})$.
\end{conjecture}

% ====== SECTION:  Quadratic density  ============================
\section{The quadratic density of constellations}\label{etaSection}

In this section we establish conditions under which the expected number of surviving constellations increases from one stage of the
sieve to the next. 

We first have to specify what meaningful measure is increasing.  For any constellation $s$ of length $J$, the populations $n_{s,J}(\pml{p})$
and $w_{s,J}(\pml{p})$ are non-decreasing.  However, the population $n_{s,J}(\pml{p})$ grows by factors of $(p-J-1)$ while the overall length
of $\pgap(\pml{p})$ grows by factors of $(p-1)$.  So the density of $s$ in $\pgap(\pml{p})$ decreases eventually to $0$.  On the other hand, the
horizon of survival $H(p)$ grows as $p^2$.

Focusing on the interval of survival $\Delta H(p_k)$, we introduce the expected number of surviving constellations $s$ of length $J$ at the
stage of Eratosthenes sieve corresponding to the prime $p_k$.  We denote this expected value as $E_s(p_k)$.  Assuming a uniform distribution
of instances of $s$ in $\pgap(\pml{p_k})$ we have
\begin{eqnarray*}
 E_s(p_k) & = &  n_{s,J}(\pml{p_k}) \cdot \frac{p_{k+1}^2-p_k^2}{\pml{p_k}} \\
   & = & n_{s,J}(p_k^\#) \cdot \frac{g \cdot (2 p_k+ g)}{\pml{p_k}}  \\
   & = & w_{s,J}(p_k^\#) \cdot \prod_{2}^{p_k}\frac{(q-J-1)_{+}}{q} \cdot g \cdot (2p_k+g)
\end{eqnarray*}
in which $g$ is the gap from $p_k$ to $p_{k+1}$:  $g= p_{k+1}-p_k$.

Over any interval of primes, these estimates $E_s(p_k)$ will vary primarily by the gaps $g$ that separate the successive primes in that interval.

The estimates are proportional to the values of $g$, and the variations in the gaps $g$ across a constellation among primes will produce
proportionate variations in the estimates $E_s(p_k)$.  These variations are color-coded in the counts displayed in Figure~\ref{DelHFig}.
Bright red marks a gap $g=2$ from $p_k$ from $p_{k+1}$, rose marks a gap $4$, bright blue marks gaps of $6$, sky blue marks gaps
of $12$.  Gold marks the samples for $\Delta H(p_k)$ with a gap of $30$ from $p_k$ to $p_{k+1}$.

In the second chart in Figure~\ref{DelHFig} for example, we are enumerating the gaps between primes in the intervals of survival
$\Delta H(p)$ as $p$ runs through the $7$ consecutive primes from $p_k=23826527$ to $p_k=23826587$.  For our uniform estimates
$E_g(p_k)$ for gaps, the populations $n_g(p^\#)$ and total span $p^\#$ introduce a factor roughly $\frac{p-2}{p}$.  Across these $7$ samples
this factor is very close to $1$.  For these $7$ consecutive primes the factor $(2p_k+g)\sim 2p$ to first order. 

In contrast, over these $7$ consecutive primes $p_k$, the factor $g=p_{k+1}-p_k$ fluctuates from $2$ through $4$, $8$, $12$, $28$, 
$6$, and another $12$.  The fluctuations in this factor $g$ are the dominant variation in the uniform estimate $E_s(p_k)$.

\vskip 0.1in

To adjust for this variation by the gap $g= p_{k+1}-p_k$, we define the {\em quadratic density} $\eta_s(p_k)$ to be a normalized version of
$E_s(p_k)$. 
$$ \eta_s(p_k) \; = \; E_s(p_k) / g. $$
$\eta_s(p_k)$ provides the average expected number of instances of $s$ over the $g$ quadratic intervals $[n^2, (n+1)^2]$ within
$\Delta H(p_k) = [p_k^2, \;  p_{k+1}^2]$.
$$ \left[p_k^2, \; (p_k+1)^2\right], \;  \left[(p_k+1)^2, \; (p_k+2)^2\right], \; \ldots \; ,\left[(p_k + g-1)^2, \; (p_k+g)^2\right].$$
We call $\eta_s(p_k)$ the quadratic density because it provides the expected number of instances of $s$ over a unit quadratic
interval $[n^2, (n+1)^2]$, when $p_k$ is large.

The quadratic density $\eta$ provides a connection between these studies and Legendre's conjecture, that there is
always a prime number between $n^2$ and $(n+1)^2$.  The Prime Number Theorem already provides an expected number of primes in
this interval.  $\eta$ provides a profile of the expected populations of various gaps and constellations in this interval.

In Lemma~\ref{LemE} below, we identify the conditions under which $\eta_s(p_k)$ is non-decreasing.

\vskip 0.1in

We first note that the populations and relative populations of any sufficiently small constellation are nondecreasing.

\begin{lemma}\label{nwLem}
Let $s$ be a constellation in $\pgap(\pml{p_k})$ of length $J < p_{k+1}-2$ and span $|s| < 2p_{k+1}$.  Then
\begin{eqnarray*}
n_{s,J}(\pml{p_{k+1}}) & \ge & n_{s,J}(\pml{p_k}) \\
w_{s,J}(\pml{p_{k+1}}) & \ge & w_{s,J}(\pml{p_k}) 
\end{eqnarray*}
\end{lemma}

\begin{proof}
For $n_{s,J}(\pml{p})$,
\begin{eqnarray*}
n_{s,J}(\pml{p_{k+1}}) - n_{s,J}(\pml{p_k}) & = & (p_{k+1}-J-1) n_{s,J}(\pml{p_k}) + n_{s,J+1}(\pml{p_k}) - n_{s,J}(\pml{p_k}) \\
   & = & (p_{k+1}-J-2) n_{s,J}(\pml{p_k}) + n_{s,J+1}(\pml{p_k})  
\end{eqnarray*}
When $p_{k+1} > J+2$, this sum is always non-negative, and it is positive once either $n_{s,J}(\pml{p_k}) > 0$ or $n_{s,J+1}(\pml{p_k}) > 0$.

For $w_{s,J}(\pml{p})$,
\begin{eqnarray*}
w_{s,J}(\pml{p_{k+1}}) - w_{s,J}(\pml{p_k}) & = &  w_{s,J}(\pml{p_k}) + \frac{\lil p_{k+1}-J-2}{\lil p_{k+1}-J-1} \cdot w_{s,J+1}(\pml{p_k}) - w_{s,J}(\pml{p_k}) \\
   & = & \frac{\lil p_{k+1}-J-2}{\lil p_{k+1}-J-1} \cdot w_{g,J+1}(\pml{p_k})  
\end{eqnarray*}
When $p_{k+1} > J+2$, this sum is always non-negative, and it is positive once $w_{s,J+1}(\pml{p_k}) > 0$.
\end{proof}

The dynamics of Eratosthenes sieve are fair.  No admissible constellation is removed or diminished.  In fact \cite{FBHPatterns, FBHktuple}, the
population of every admissible constellation of length $J$ grows by the same factor $(p-J-1)$.

The above proof shows that for any constellation $s$ of length $J$, if $s$ has driving terms of length $J+1$, then $w_{s,J}(\pml{p})$ is increasing.
If $s$ does not have any driving terms of length $J+1$, then $w_{s,J}(\pml{p})$ is constant.  This establishes that the graphs of any
$w_{s,J}(\lambda)$ are constant or convex.  

In Figure~\ref{AllGapsFig} we plot the curves for the relative populations of the gaps $g \le 82$.
Since we take our initial conditions from $\pgap(\pml{37})$, we are only able to determine the complete models for constellations up to this bound.
Lemma~\ref{nwLem} assures us that the graphs of the $w_{s,J}(\lambda)$ will have similar shapes to those depicted in Figure~\ref{AllGapsFig}.

Figures~\ref{ModelsJ2Fig} and~\ref{ModelsJ3Fig} show the relative population models for a selection of constellations of length $J=2$ and $J=3$
respectively.

% == subsection - eta ========================
\subsection{In the large, quadratic density $\eta$ is non-decreasing.}
We are ready to establish the conditions under which $\eta_s(p_k)$ is non-decreasing.
We consider the ratio $\eta_s(p_k) / \eta_s(p_{k-1})$ and determine conditions under which this ratio is at least $1$.

To set this up, we start in $\pgap(\pml{p_{k-2}})$.  As illustrated in Figure~\ref{GpkFig}, in $\pgap(\pml{p_{k-2}})$ the first gap
$g_1 = p_{k-1}-1$.  The second gap $g_2 = p_k - p_{k-1}$, and the third gap $g_3 = p_{k+1}-p_k$.

\begin{figure}[hbt]
\centering
\includegraphics[width=2.25in]{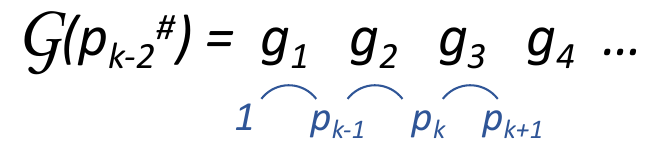}
\caption{\label{GpkFig} To compare the estimates of survival from the intervals of survival $\Delta H(p_{k-1})$ and $\Delta H(p_{k})$, 
we look at the first few gaps in the cycle $\pgap(\pml{p_{k-2}})$.}
\end{figure}

% LEMMA on eta non-decreasing
\begin{lemma}\label{LemE}
Let $s$ be a constellation of length $J$ such that the span $|s| < 2p_1$ and $n_{s,J}(\pml{p_0}) > 0$.
For $k \ge 2$ let ${g_2 = p_k - p_{k-1}}$ and ${g_3 = p_{k+1}-p_k}$.

\noindent The ratio $\frac{\eta_s(p_k)}{\eta_s(p_{k-1})} \ge 1$ iff
\begin{equation}
g_2 + g_3 \; \ge \; (J+1 - \delta_s) \cdot (2 + \frac{g_3}{p_k})  \label{EqGEnu}
\end{equation}
where $\delta_s = \frac{n_{s,J+1}(\pml{p_{k-1}})}{n_{s,J}(\pml{p_{k-1}})}$.
\end{lemma}

\begin{proof}
\begin{eqnarray*}
\frac{\eta_s(p_k)}{\eta_s(p_{k-1})} & = & \frac{E_s(p_{k})/ g_3}{E_s(p_{k-1})/g_2} \; = \; \frac{n_{s,J}(\pml{p_k})}{n_{s,J}(\pml{p_{k-1}})} \cdot \frac{1}{p_k}
  \cdot \frac{g_3 (p_{k+1}+p_k)/g_3}{g_2 ( p_k+p_{k-1})/g_2} \\
  & = &  \frac{1}{p_k} \cdot \frac{ (p_{k+1}+p_k)}{( p_k+p_{k-1})} \cdot \frac{n_{s,J}(\pml{p_k})}{n_{s,J}(\pml{p_{k-1}})} \\
  & = & \frac{1}{p_k} \cdot \frac{ (p_{k+1}+p_k)}{( p_k+p_{k-1})} \cdot \frac{(p_k-J-1) n_{s,J}(\pml{p_{k-1}})+n_{s,J+1}(\pml{p_{k-1}})}{n_{s,J}(\pml{p_{k-1}})}  \\
  & = & \frac{1}{p_k} \cdot \frac{ (p_{k+1}+p_k)}{( p_k+p_{k-1})} \cdot \left[ (p_k-J-1) + \frac{n_{s,J+1}(\pml{p_{k-1}})}{n_{s,J}(\pml{p_{k-1}})} \right] \\
  & = & \frac{1}{p_k} \cdot \frac{ (p_{k+1}+p_k)}{( p_k+p_{k-1})} \cdot \left[ (p_k-J-1) + \delta_s \right] 
\end{eqnarray*}
Thus $\frac{\eta_s(p_k)}{\eta_s(p_{k-1})} \ge 1$ iff
\begin{eqnarray*}
\frac{1}{p_k} \cdot \frac{ (p_{k+1}+p_k)}{( p_k+p_{k-1})} \cdot \left[ (p_k-J-1) + \delta_s \right] & \ge & 1 \\
\frac{ (p_{k+1}+p_k)}{( p_k+p_{k-1})} - 1 & \ge & \frac{1}{p_k} \cdot \frac{ (p_{k+1}+p_k)}{( p_k+p_{k-1})} \cdot \left[ J+1 - \delta_s \right]  \\
 (p_{k+1} -p_{k-1}) & \ge & \frac{1}{p_k} \cdot (p_{k+1}+p_k) \cdot \left[ J+1 - \delta_s \right]  \\
 g_2 + g_3 & \ge & \frac{1}{p_k} \cdot (2 p_k + g_3) \cdot \left[ J+1 - \delta_s \right]  \\
g_2 + g_3 & \ge & (J+1 - \delta_s) \cdot (2 + \frac{g_3}{p_k})
\end{eqnarray*}
This is the desired Equation~(\ref{EqGEnu}).
\end{proof}

The parameter $\delta_s$ is the ratio of instances of driving terms for $s$ of length $J+1$ 
to instances of $s$ itself, length $J$, in $\pgap(p_{k-1}^\#)$.

Note that in Lemma~\ref{LemE} the constraint~(\ref{EqGEnu}) holds if and only if $\eta_s$ is nondecreasing at $p_k$.  Expressed as it is in
Equation~(\ref{EqGEnu}) the constraint ties the behavior of $\eta_s$ simply and explicitly to the succession of gaps $g_2, g_3$, to the length
$J$, qualified surprisingly by the ratio $\delta$, and to the ratio $\frac{g_3}{p_k}$.

The lefthand side of the constraint~(\ref{EqGEnu})  is the span $g_2+g_3$ of the constellation $g_2 g_3$ from $p_{k-1}$ to $p_{k+1}$.  See Figure~\ref{GpkFig}.
This is the only occurrence of the gap $g_2$ in the constraint.  If this span $g_2+g_3$ is above the bound given by the righthand side of the constraint,
 then $\eta_{s,J}$ increases at $p_k$, and otherwise $\eta_{s,J}$ decreases at $p_k$.

The smallest constellations of length $2$ are $s=24$ and $s=42$, each of which has span $|s|=6$, so ${g_2+g_3 \ge 6}$.

\begin{theorem}\label{gapThm}
For every gap $g$, once $g$ occurs in $\pgap(\pml{p_0})$ for any $p_0\ge 3$, the quadratic density $\eta_g$ increases for all primes $p_k > p_1$.
\end{theorem}

\begin{proof}
For a gap $g$, the length $J=1$.  The parameter $\delta_g \ge 0$, and ${0 < \frac{g_3}{p_k} < 1}$. So 
$$ g_2 + g_3 \ge 6 > (2-\delta_g)\cdot (2 + \frac{g_3}{p_k}), $$
and the inequality in (\ref{EqGEnu}) always holds.  Thus $\eta_g$ is increasing.
\end{proof}

That is, for every gap $g$, if $g$ occurs in $\pgap(\pml{p})$, the expected number of occurrences of $g$ in the interval of survival $\Delta H(p_k)$, 
proportionate to $g_3$, increases for all $p_k > p$.   Under the dynamics of the system, the quadratic density of every gap increases.

For $J \ge 2$, results along the lines of Theorem~\ref{gapThm} come with qualifications.  
For example, for constellations $s$ of length $J = 2$, the constraint~(\ref{EqGEnu})
will hold for large $p_k$ when $g_2 + g_3 \ge 8$, or equivalently except when $g_2g_3 = 24$ or $42$.

On the other hand, $g_2+g_3$ is the most volatile part of the constraint (\ref{EqGEnu}).  Its minimum value is $6$, but the value $g_2+g_3$ 
fluctuates wildly.  

The results will also be sensitive to the value of $\delta_s= \frac{n_{s,J+1}(\pml{p_k})}{n_{s,J}(\pml{p_k})}$, which depends both on $s$ and $p_k$.  We will return to an analysis of $\delta_s$
later, but for the time being we ignore it.

We define
\begin{equation} 
f(p_k, J) \; = \; \frac{g_2+g_3}{2+ g_3/p_k} - (J+1). \label{Eqf}
\end{equation}
This function $f$ is linear in $J$, and $\delta_s$ would be added to it to provide a complete equivalent to the constraint~(\ref{EqGEnu}). 

\begin{lemma}\label{fLemma}
Suppose the admissible constellation $s$ of length $J$ occurs in the cycle of gaps $\pgap(p_0^\#)$.
Then for $k \ge 2$,
\begin{eqnarray*}
 f(p_k,J)+\delta_s \; > \; 0 & {\rm iff} &  g_2 + g_3 \; > \; (J+1 - \delta) \cdot (2 + \frac{g_3}{p_k}) \\
  & {\rm iff} & \eta_s(p_k) \; > \; \eta_s(p_{k-1})
 \end{eqnarray*}
 \end{lemma}

Since $\delta_s\ge 0$, when $f (p_k,J) > 0$ the constraint~(\ref{EqGEnu}) holds for {\em all} admissible 
constellations $s \subset \pgap(\pml{p_k})$ of length $J$.
When $f(p_k,J) <0$, the constraint is not met -- unless the addition of $\delta_s$ is sufficient to change the sign.  So $f(p_k,J)$ provides a surrogate for the constraint~(\ref{EqGEnu}) as $\delta_s \fto 0$.

$f$ captures the part of the constraint~(\ref{EqGEnu}) that does not vary with the specific constellation $s$, and it isolates the contribution of $s$ as 
a simple offset $+\delta_s$.  Figure~\ref{fpjFig} shows six samples of $f(p_k,J)$ over consecutive primes within the first million primes.  
For these samples we show ${1 \le J \le 16}$.
Across each row, $f(p,J)$ is linear in $J$.  

% figure for f(p,J) sample
\begin{figure}[hbt]
\centering
\includegraphics[width=5.125in]{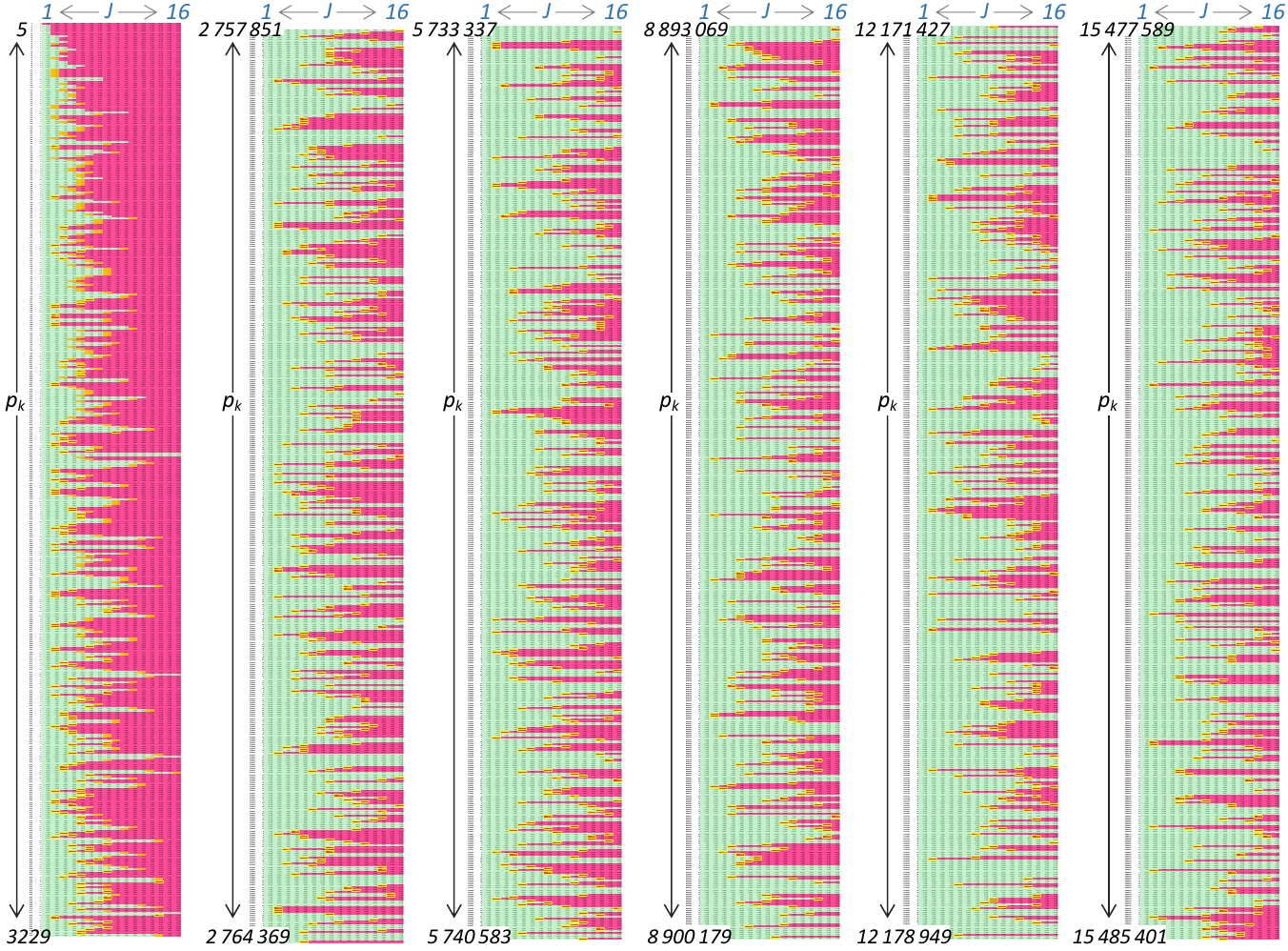}
\caption{\label{fpjFig} Six samples of $f(p_k,J) =\frac{g_2+g_3}{2+g_3 / p_k}-(J+1)$ over the first million primes.  Each strip illustrates the sign of
$f(p_k,J)$ over consecutive primes.  The green background shows where $f$ is positive
and $\eta$ is increasing for all $s$ of length $J$.  The magenta shows where $f$ is negative and $\eta$ generally decreases, 
although we still need to apply the shift by $\delta_s$.}
\end{figure}

In Figure~\ref{fpjFig} we observe the volatility of the span $g_2+g_3$.
Each row turns from green to red when
the basic constraint (ignoring $\delta$ which is specific to $s$) no longer holds for $J$.
Each strip is a sample of consecutive primes.  Within each strip, each row corresponds to a single prime $p$, and the graph
shows the sign-change of $f(p,J)$ in $J$.  Where the graph is green, the quadratic density increases for all admissible constellations
of length $J$ from this cycle $\pgap(\pml{p})$ to the next.  Where the graph turns red, there is a threshold, linear in $J$, that $\delta_s$
would have to exceed in order for the quadratic density of $s$ to increase.

Using Figure~\ref{fpjFig} we make two observations about $f(p_k,J)$.  
First, for $J \ge 2$ the constraint~(\ref{EqGEnu}) regularly fails (for small values of 
$\delta_s$).  

Our second observation from Figure~\ref{fpjFig} is that growth in the quadratic density $\eta_s$ significantly 
lags the basic survival of constellations in $\pgap(\pml{p})$.
For an admissible constellation $s$ of length $J$, once $p_k > J+1$ the population $n_{s,J}(\pml{p_k})$ generally grows by factors of ${p_k-J-1}$.
With the horizon of survival being the square $p^2_{k+1}$, over the first million primes we might expect the survival of constellations of lengths up to $J=3935$
to have stabilized.  Instead, Figure~\ref{fpjFig} shows that constellations with $J >8$ still do not have consistent growth in their quadratic density.

When $g_2+g_3$ is large, the horizon of survival $p^2_{k+1}$ advances further and more constellations are preserved.  When $g_2+g_3$ is
small, the horizon of survival does not advance as far and long constellations are at risk of further fusions.

% figure for f(p,J) sample
\begin{figure}[hbt]
\centering
\includegraphics[width=5.125in]{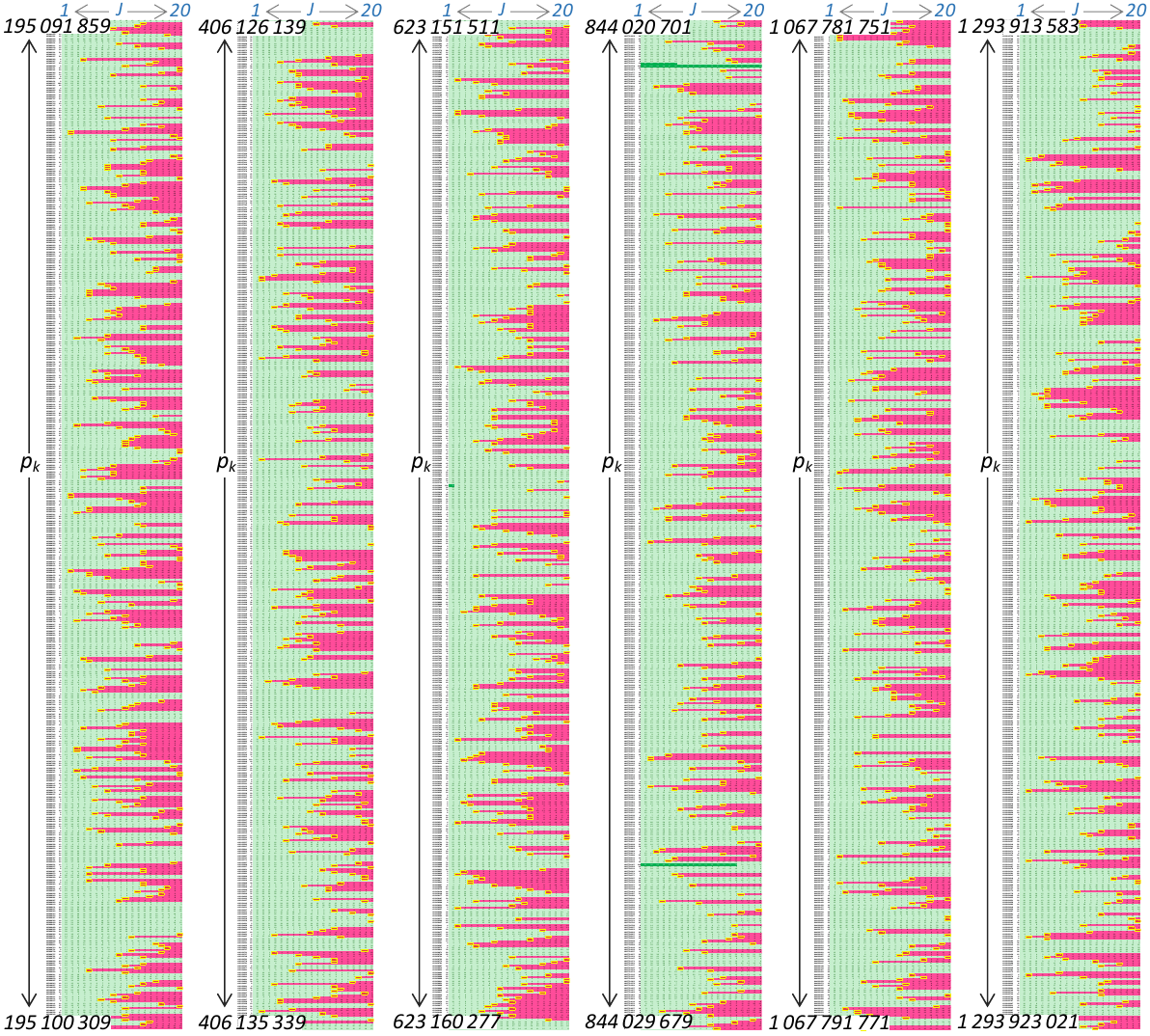}
\caption{\label{fpjBFig} Six additional samples, taken across the first $64$~million primes, of ${f(p_k,J) =\frac{g_2+g_3}{2+g_3 / p_k}-(J+1)}$.
Each strip illustrates the sign of
$f(p_k,J)$ over consecutive primes.  The green background shows where $f$ is positive
and $\eta$ is increasing for all $s$ of length $J$.  The magenta shows where $f$ is negative and $\eta$ generally decreases, 
although we still need to apply the shift by $\delta_s$.}
\end{figure}

Figures~\ref{fpjFig}~and~\ref{fpjBFig} show how the volatility of $g_2+g_3$ dominates the constraint~(\ref{EqGEnu}).
The righthand side of the constraint is much more stable.  Since the lefthand side of (\ref{EqGEnu}) is always an integer, the value of the righthand
side has no effect until it crosses an integer value.

\subsection{Regarding $g_3/p_k$.}
The second factor on the righthand side of (\ref{EqGEnu}) is $(2+ g_3/p_k)$ 
where ${g_3 = p_{k+1}-p_k}$.  
The integer value of the righthand side of (\ref{EqGEnu}) will not increase beyond $2 \cdot (J+1-\delta)$ unless 
$\frac{g_3}{p_k} \ge \frac{1}{J+1-\delta}$.

From the prime number theorem we know that
${g_3/p_k \fto 0}$, with ${g_3/p_k < \frac{1}{16597}}$ for $p_k > 2010760$ (Theorem~12 of \cite{Sch76}).  
For primes ${p_k > 2010760}$, unless we are considering 
constellations of length $J > 16597$,
the factor ${(2+ g_3/p_k)}$ does not increase the righthand side of the constraint by $1$, and thus does not change the range of $g_2+g_3$ that 
satisfy the constraint~(\ref{EqGEnu}).

\subsection{Regarding $\delta_s$.}
The ratio $\delta_s$ is the only parameter in the constraint~(\ref{EqGEnu}) that varies across the admissible constellations of length $J$.
 
 The first impact of having 
 $$ \delta_s = \frac{n_{s,J+1}(p_{k-1}^\#)}{n_{s,J}(p_{k-1}^\#)}$$
 in the constraint is to require that $n_{s,J}(p_{k-1}^\#) > 0$.  That is, the constellation $s$ itself must actually occur in
 $\pgap(p_{k-1}^\#)$.   
 
 Once $s$ occurs in the cycle, then $\delta_s$ is the ratio of driving terms for $s$ of length $J+1$ to
 the population of $s$ itself, length $J$.
 The value of $\delta_s$ depends both on $s$ and on $p_k$.  Early in the evolution of $s$, we often have 
$$ n_{s,J+1}(p_k^\#) > n_{s,J}(p_k^\#),$$
more driving terms for $s$ of length $J+1$ than instances of $s$ itself.

Figure~\ref{dgapFig} reminds us of the scale of ``early'' in the evolution of these populations.  The curves in Figure~\ref{dgapFig}
are not stacked; we colored the spaces between the curves for legibility.  The figure shows $\delta_{J=3}(\lambda)$, where the
curves aggregate the values for constellations of length $J=3$ by span, for spans
$6 \le |s| \le 82$.  As the primes $p \fto \infty$ the ratio $\delta_{J=3}(\lambda)$ does indeed go to $0$, but even for modest spans,
$\delta_s > 1$ well into the evolution across $p_k$.  

% XXXQHERE [2/14] - Evolution of delta
In Figure~\ref{dgapFig} we see that for this range of spans $|s| \le 82$ the ratio $\delta_{J=3}(\lambda)$ doesn't pass below $1$
for a long time.  For example, the span $|s|=60$ has $\delta_{|s|=60}(\lambda) > 1$ up through primes $p \sim 8.8E4$, and 
the intervals of survival $\Delta H(p)$ start at $p^2 \sim 7E9$.  Before these thresholds, the factor $(2-\delta_g)$ in the constraint 
in Equation~\ref{EqGEnu} is weakened.  The larger
result of Lemma~\ref{LemE} addresses the quadratic density of gaps between very large primes.

For a while $\delta_s$ may be large, but $\delta_s \fto 0$ as $p_k \fto \infty$.
From the eigenstructure \cite{FBHPatterns, FBHSFU} of the discrete dynamic system for $\pgap(\pml{p_k})$,
\begingroup
\renewcommand{\arraystretch}{2}
$$
\begin{array}{ll}
 \delta_s(p_k) & =  \frac{n_{s,J+1}(\pml{p_k})}{n_{s,J}(\pml{p_k})} \\
  & = 
  \frac{l_2 \cdot \lambda - 2 l_3\cdot \lambda_3 + 3 l_4 \cdot \lambda_4 - \cdots + (-1)^{J_1-J+1}(J_1-J)l_{J_1-J+1} 
  \cdot \lambda_{J_1-J+1}}{l_1 - l_2 \cdot \lambda  + l_3 \cdot \lambda_3  - \cdots + (-1)^{J_1-J}  l_{J_1-J+1} \cdot \lambda_{J_1-J+1}}
 \end{array}
 $$
 \endgroup
 where $J_1$ is the length for the longest driving term for $s$;
 and the coefficients $l_j$ and eigenvalues $\lambda_j$ are as described above.

The parameter $\lambda=\lambda_2$ decays to $0$ slowly with $p_k$.
$$1 > \lambda_2 > \lambda_3 > \cdots > \lambda_{J_1-J+1}$$
with $\lambda_k \approx \lambda_2^{k-1}$.
No matter what the initial conditions, for all $s$ we know 
that the ratio $\delta_s \ge 0$, and $\delta_s \fto 0$ slowly with $p_k$.

\begin{figure}[hbt]
\centering
\includegraphics[width=5in]{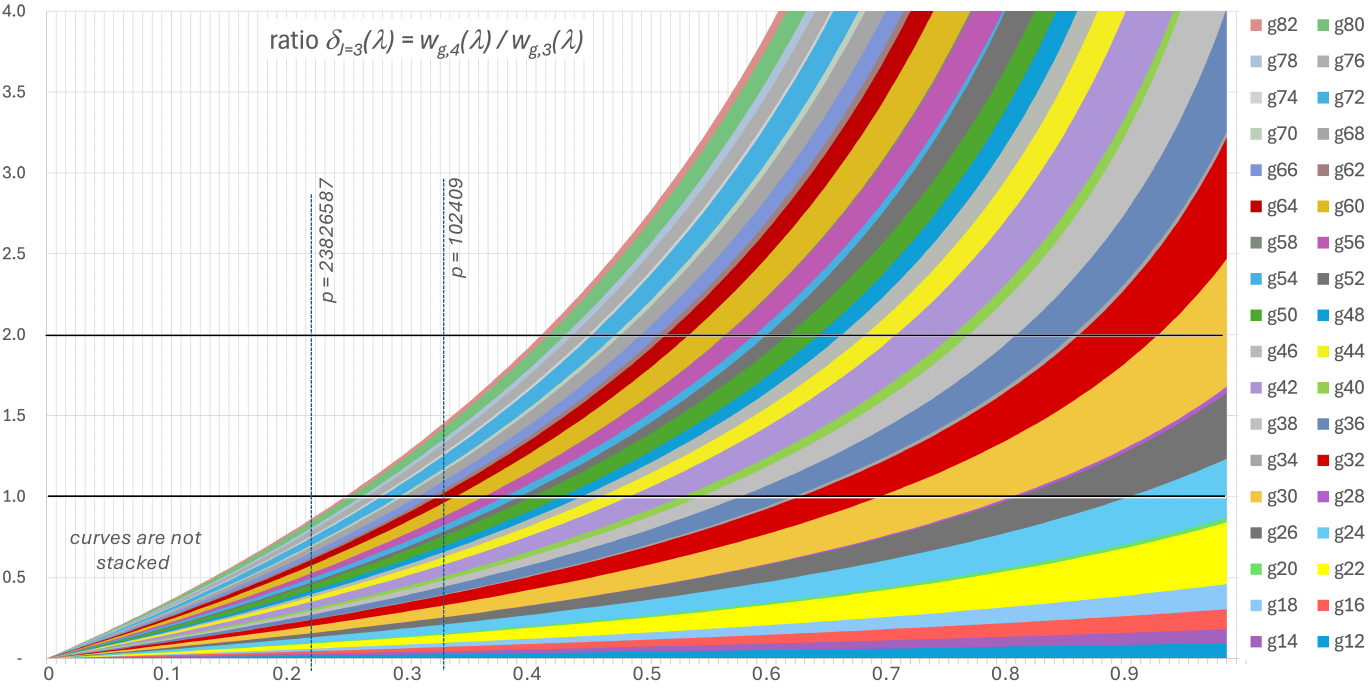}
\caption{\label{dgapFig} 
The graphs for $\delta_{J=3}(\lambda)$ for constellations of length $J=3$ and spans of $s$ in ${12\le |s| \le 82}$.  
The graphs of each $\delta_{J=3}(\lambda)$ 
are not stacked and are colored consistently with Figure~\ref{AllGapsFig}.  The counts for individual constellations are aggregated here by span. }
\end{figure}

% === fragment  XXXQHERE === fragment ====
% For $\delta_s$ the behavior is slightly more complicated.  What more can we learn about $\delta_s$?
 
% As $p$ grows, this line of symmetry becomes steeper, and for any fixed $(g,J)$ eventually $\delta_{g,J} < 1$.
% This means that for large primes Equation~(\ref{EqGEnu}) is approximated by the condition
% $$ g_2 + g_3 \ge 2(J+1). $$

% \begin{lemma}
% Let $s$ be an admissible constellation of length $J$ and span $g = |s|$.
% Then there is a prime $q$ with $J > \frac{\phi(\pml{q})}{\pml{q}} g$, and for all primes $p_k > q$, $\delta_s < 1$.
% \end{lemma}

\FloatBarrier

% == subsection:  short constellations
\subsection{Samples of short constellations.}
We have models and samples for a few short constellations of small span.
Figure~\ref{ModelsJ2Fig} shows the models $w_{s,2}(\lambda)$ for a dozen constellations of length $J=2$, the section at 
$\lambda=0.36$, and counts of occurrences among primes across corresponding $\Delta H(p)$.
We begin to see the noise that the factor $(g_2+g_3)$ introduces into survival.  The models are still good first-order estimates, but the
errors are significantly larger than those for gaps.

As a technical note, the subdominant eigenvalue $\lambda$ that we use as the system parameter varies by length $J$.
$$ \lambda(J) \; = \; \prod_{p_1}^{p_k} \frac{q-J-2}{q-J-1}$$
While the values of the various $\lambda(J)$ are close to each other, we cannot quite align the graphs of the $w_{s,J}(\lambda)$
as the length $J$ varies.

\begin{figure}[hbt]
\centering
\includegraphics[width=5.125in]{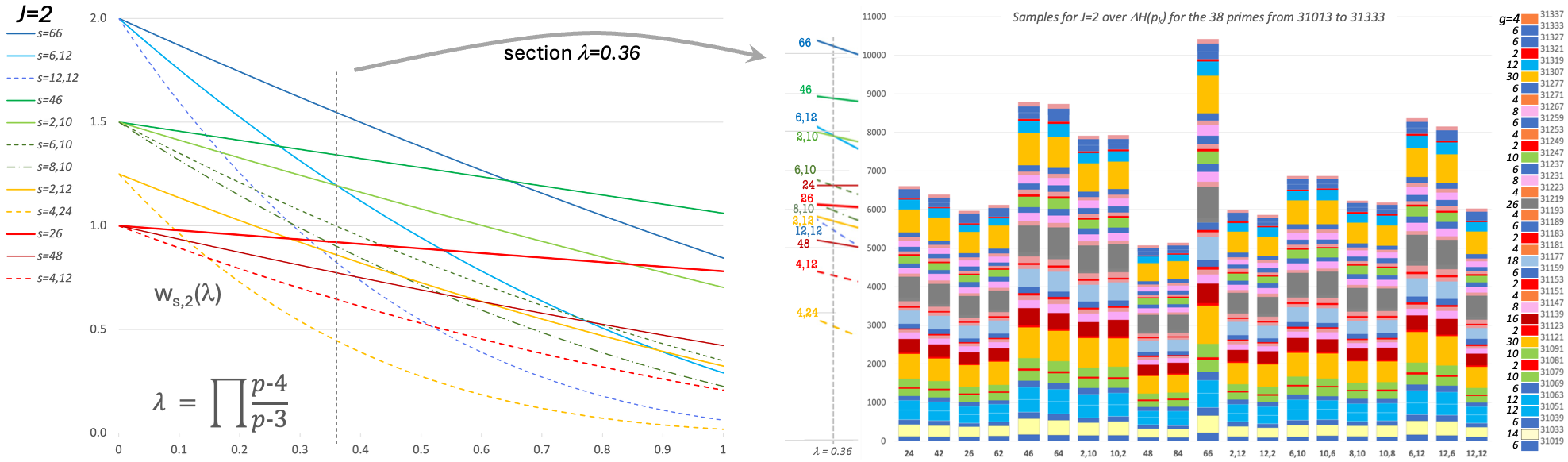}
\caption{\label{ModelsJ2Fig} 
Models of the relative populations $w_{s,2}(\pml{p})$ for a few constellations $s$ of length $J=2$ and small span.
Their asymptotic values $w_{s,2}(\infty)$ depend on whether some span between boundary fusions is divisible by $5$, $7$, or $11$.
The righthand graph compares the section at $\lambda=0.36$ with the counts of occurrences among primes in $\Delta H(p)$ for
${p\sim 31019}$.}
\end{figure}

The righthand panel in Figure~\ref{ModelsJ2Fig} 
shows the accumulated populations for a few constellations of length $J=2$, within the intervals of survival
$[p_k^2, p_{k+1}^2]$ for the $38$ primes from $p_k=31013$ to $31333$.  
For example the first bar depicts the populations of the constellation $s=24$ within these intervals.  Each interval is color-coded by
the gap $g=p_{k+1}-p_k$, consistent with Figure~\ref{DelHFig}.

\begin{figure}[hbt]
\centering
\includegraphics[width=5.125in]{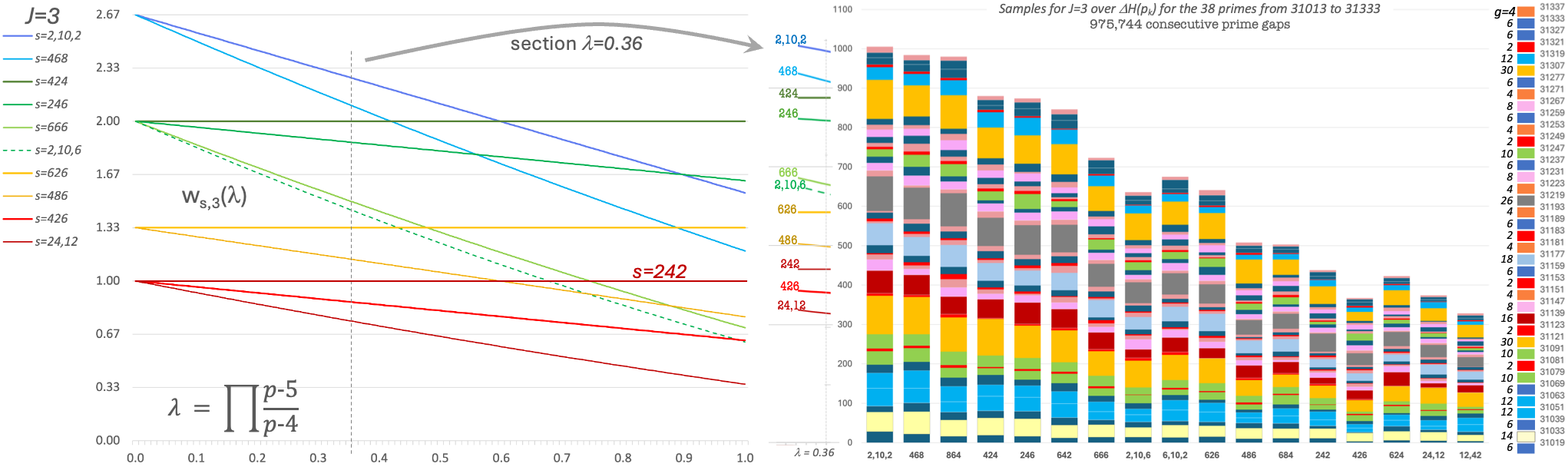}
\caption{\label{ModelsJ3Fig} 
Models of the relative populations $w_{s,3}(\pml{p})$ for a few constellations $s$ of length $J=3$ and small span.
The asymptotic values $w_{s,3}(\infty)$ for this selection of constellations depend on whether some span 
between boundary fusions is divisible by $5$ or $7$.
Actual populations of the constellations of length $J=3$ among those with the highest relative populations for 
$\lambda = 0.356$.  These samples are taken over $\Delta H(p_k)$ for $38$ consecutive primes, from $p_k=31013$ to $31333$.
The counts follow the relative population models $w_{s,3}(0.36)$, shown at left, to first order.}
\end{figure}

Similarly the righthand panel in Figure~\ref{ModelsJ3Fig} shows the accumulated populations of a sample of constellations of
length $J=3$ across $\Delta H(p)$ for $38$ consecutive primes from $p=31013$ to $31333$.  The relative population models
$w_{s,3}(\lambda)$ for these constellations are shown on the left.  There is general first-order agreement with the uniform estimate, but 
we observe larger deviations.  

Interestingly, for these samples the symmetric pairs of constellations -- e.g. $s=246$ and $s=642$, or $s=2,10,6$ and $s=6,10,2$, 
or $s=426$ and $s=624$, or $s=24,12$ and $s=12,42$ -- exhibit noticeably unequal survival, although their populations in the 
larger cycle $\pgap(p^\#)$ are exactly equal.

% XXXQHERE [2/17] - J=5 & Github.

% XXXQHERE [2/16] - relate back to figures {fpjFig} and {fpjBFig}

% ====== SECTION:  Discussion ==============================
\section{Discussion of expected survival}\label{SurvivalSection}
Gaps cannot occur among primes unless they arise in Eratosthenes sieve and survive the ongoing fusion process, passing
under the horizon of survival to be confirmed as gaps among primes.

The analysis presented above, regarding the survival of the gaps in the cycle $\pgap(\pml{p_k})$ as gaps between primes,  
relies on an hypothesis which we stated above as Conjecture~\ref{UniConj}.

Across primes $p_k$ the behavior of the expected survival of $s$
$${E_s(p_k) = n_{s,J}(\pml{p_k})\cdot g_3 \cdot (2p_k+g_3)/ \pml{p_k}}$$
is dominated by the factor ${g_3 = p_{k+1}-p_k}$.  

Over a set of several consecutive primes, the factor $g_3$ creates proportionate fluctuations in $E_s(p_k)$.
These fluctuations are apparent in the samples displayed in Figure~\ref{DelHFig}.
They appear as the intra-column variations in the accumulated populations of each gap within the intervals of survival $\Delta H(p_k)$,
which are color-coded by the size of the gap ${g_3=p_{k+1}-p_k}$.

Consequently we have defined the quadratic density $\eta_s(p_k) = E_s(p_k)/g_3$.  $\eta_s(p_k)$ is the average occurrence of $s$ in 
the $g_3$ intervals $[n^2, (n+1)^2]$ across the interval of survival $\Delta H(p_k)$.  We expect that $\eta_s(p_k)$ should
be more stable over consecutive primes.

Figure~\ref{DelHFig} shows the accumulated populations of gaps across intervals of survival $\Delta H(p)$
for ranges of consecutive primes corresponding to four distinct values of $\lambda$.
The broken line indicates the estimates from relative population models from Figure~\ref{AllGapsFig},
at the corresponding values of $\lambda$.

In our samples of gaps between primes, represented by the examples in Figure~\ref{DelHFig}, 
we see no large variations away from the uniformity predicted by Conjecture~\ref{UniHConj}, either across the column heights or
across the bands within each column.  A large deviation in the column heights would indicate that the surviving populations of gaps
strongly prefer (i.e. a column significantly higher than the estimates) or deprecate (i.e. a column significantly lower) specific gaps.  

Any large deviations among the bands of a single column, 
beyond the proportionality to $g=p_{k+1}-p_k$, would indicate that the populations of the gap corresponding to this column are
clustered non-uniformly across the intervals of survival $\Delta H(p_k)$.

To strengthen our understanding of the deviations among the bands in a single column, we have defined the quadratic
density $\eta_s(p)$.  From our samples, represented here by Figure~\ref{EtaSampleFig}, we observe mean values of $\eta_g$
consistent with Conjecture~\ref{UniHConj} and low variance.

Additional samples are available in the Jupyter notebooks on our GitHub site, and the code enables creating and exploring 
additional samples.

% == Subsection - Legendre's Conjecture =========
\subsection{Regarding Legendre's Conjecture.}
In $\pgap(\pml{p_k})$ there are gaps at least as large as $g = 2p_{k-1}$.  The maximum gap $g_{\max}(p_k^\#)$ in the cycle
$\pgap(\pml{p_k})$ defines Jacobsthal's function \cite{Hag09}.  The value of $g_{\max}(p_k^\#)$ quickly exceeds $2p_{k-1}$.
Here is a sample of values of $g_{\max}(p_k^\#)$ from \cite{Hag09}:

\begin{center}
\begin{tabular}{r|cccccccccc}\hline
$p_k$                  & $23$ & $37$ &  $67$ & $103$ & $107$  & $149$ & $151$ & $211$ & $223$  & $227$ \\ \hline
$g_{\max}(p_k^\#)$ & $40$ & $66$ & $152$ & $282$ & $300$ & $432$ & $450$ & $686$ & $718$ & $742$ \\
\end{tabular}
\end{center}

These large gaps arise but have small relative populations in $\pgap(\pml{p_k})$.  Statistically they are unlikely to appear in the interval of 
survival.  However, if a gap $g > 2p_k$ did fall within $\Delta H(p_k)$, this would challenge Legendre's conjecture.

Consider a counterexample to Legendre's conjecture.  The counterexample is an integer $n$ and a gap $g$ such that $g$ spans 
the interval $[n^2, (n+1)^2]$.  Since both squares are composite, a gap in $\Delta H(p_k)$ that spanned this interval, preventing a
prime from occurring, would have to satisfy $g \ge 2n+4$, for some $n$ in ${p_k \le n < p_{k+1}}$.

Thus gaps $g \ge 2p_k+4$ in $\pgap(p_k^\#)$ present possible challenges to Legendre's conjecture.  These large gaps would have
to occur within the interval of survival $\Delta H(p_k)$ {\em and} cover an interval $[n^2,(n+1)^2]$.  However, large gaps ${g \sim 2p_k+4}$
could fall within the interval of survival and still fail to cover the required quadratic interval, e.g. half in one such interval and half in the
next.

If still larger gaps $g \ge 4p_k+6$ occur early in the interval of survival or gaps ${g \ge 4p_{k+1}-2}$ occur at all in $\Delta H(p_k)$, then
Legendre's conjecture is false.  These gaps are big enough that their very occurrence within $\Delta H(p_k)$ requires that some
quadratic interval $[n^2,(n+1)^2]$ is covered.  

From Jacobthal's function we know that large gaps do occur in the cycles $\pgap(p^\#)$.
If Legendre's conjecture holds, then the largest gaps $g \ge 4p+6$  have to occur beyond the horizon of survival, and
if large gaps $g$ of span $g > 2p+4$ fell within the horizon of survival, they would have to be misaligned or out of phase with
the quadratic intervals.

Any occurrence of such large gaps in the interval $\Delta H(p)$ would be a remarkable statistical anomaly.
Theorem~\ref{gapThm} establishes that the expected quadratic densities of all extant gaps in $\pgap(p^\#)$ are increasing.
Figure~\ref{EtaSampleFig} illustrates how low the variances are for the samples we have calculated.

% Updated table of values - using 0.220 as illustrated.
% XXXQHERE [2/27] - Be clear about expected value E vs measured values.
% XXXQHERE
For example, by $\lambda=0.220$ or primes $p_k \sim 23,826,527$, we have
$$
\begin{array}{rrrrrrr}
g = & 2 & 4 & 6 & 8 & 10 & 12 \\
\eta_g(p_k) = & 54498 & 54534 & 102034 & 47691 & 62544 & 86148 
\end{array}
$$
That is, at these stages of the sieve,
we expect around $54498$ gaps $g=2$ in every interval $[n^2, (n+1)^2]$, $54534$ gaps $g=4$, plus around $102,034$ gaps
$g=6$, and so on.  For small gaps, the standard deviation of samples in this range is less than $0.5\%$ of the mean values.

Again let $g_3= p_{k+1}-p_k$.  If a gap covered one of the $g_3$ quadratic intervals $[n^2, (n+1)^2]$ in $\Delta H(p_k)$, the populations
of the other gaps in $\Delta H(p_k)$ would dip on average by $1/g_3$.  Possible, but a statistical outlier contrary to the trending
quadratic densities.

\vskip 0.125in

The discussion above leads us to consider two functions related to Jacobthal's function.  Jacobthal's function, often denoted $h(p_k^\#)$
is the largest gap in the complete cycle $\pgap(p_k^\#)$.  Let's define a function over intervals ${\mathcal I}$
$$ g_{\rm max}({\mathcal I}, p) = {\rm max} \set{ g \st g \in \pgap(p^\#) \cap {\mathcal I}}.$$
In this notation Jacobthal's function $h(p^\#) = g_{\rm max}([1, p^\#+1], p)$.

We are curious about the restriction of this function to the intervals of survival
$$ g_{\rm max}(\Delta H(p), p)$$
This function is related to Jacobthal's function, but it is restricted to the interval of survival.
What are the trends of the maximal gaps $g$ that occur within $\Delta H(p)$?

The second function we're curious about is a refinement of that first one. 
$$ g_{\rm max}([(p_k+j)^2, (p_k+j+1)^2], p_k)$$
where $0 \le j < g_3$.  Of course for fixed $p_k$ the max over $j$ will yield the previous function $g_{\rm max}(\Delta H(p_k), p_k)$,
but this refinement to the individual quadratic intervals  provides more information about the outliers in the populations of gaps.

\subsection{Future directions}
There are several directions in which to continue these studies.  Of course, expanding the data samples for gaps and constellations
surviving in $\Delta H(p)$ is an ongoing need.  It would be useful too to get the coefficients for the complete models for gaps
$84 \le g \le 90$ or larger.  

Here we highlight two directions for progress on the theory underlying these studies.

One direction is to improve our models for surviving Eratosthenes sieve.  We have Conjecture~\ref{UniHConj} which postulates
approximate uniformity of the distribution of gaps $g$ in $\Delta H(p)$, weighted by their relative populations $w_{g,1}(p^\#)$.
Our samples show that uniformity provides good first-order estimates for $J=1$ and $J=3$.

For $J=5$ our samples exhibit larger deviations from the uniform estimates.  Part of this deviation is due to the much smaller actual
counts $N_{s,5}(p)$ of the sample constellations within the intervals $\Delta H(p)$.  Figures~\ref{ModelsJ2Fig} and~\ref{ModelsJ3Fig}
show that across the same intervals $\Delta H(p)$ the sample populations for $J=2$ range between $5000$ and $9000$, and for 
$J=3$ the sample populations range between $375$ and $1000$.  As $J$ increases we need longer intervals perhaps to observe the
uniformity.

The larger deviations away from the simple uniform estimates are compounded by Lemma~\ref{fLemma}, the effect of which is illustrated
in the samples in Figures~\ref{fpjFig} and \ref{fpjBFig}.  In those figures, for fixed $p$ the function $f(p, J)$ is linear in $J$. 
When constellations $g_2g_3$ of small span occur between primes, the
quadratic density dips, even for constellations of modest length (say $J>4$).

We trust the Conjecture~\ref{UniHConj} for $J \le 3$.
For larger $J$ the numbers of occurrences are small and the quadratic density can fluctuate.  There is more detail around these
structures to be discovered.

\vspace{0.1in}

Another direction for future work is a study of {\em where} large gaps ${g \ge 2p_k+4}$ occur in the cycle. 
Hagedorn has produced a few values for Jacobthal's function \cite{Hag09}, and these consistently exceed the bound $2p_k$ from 
investigating Legendre's conjecture.  We propose to develop primorial coordinates \cite{FBHPatterns, FBHktuple} for 
Hagedorn's examples \cite{Hag09}, to see how the instances of large gaps are situated within the cycles of gaps $\pgap(p_k^\#)$
relative to the horizon of survival $p_{k+1}^2$.

%=======================================

\bibliographystyle{amsplain}

% \bibliography{primes}

\begin{thebibliography}{10}

\bibitem{BrentSmall}
R.P. Brent, \emph{The distribution of small gaps between successive primes}, Math. of Computation, 28(125), Jan 1974.

\bibitem{Brown}
S. Brown, \emph{Distance between consecutive elements of the multiplicative group of integers modulo $n$}, preprint, Dec 2023.

\bibitem{Hag09}
T.R. Hagedorn, \emph{Computation of Jacobsthal's function $h(n)$ for $n < 50$}, Math. of Computation, 78(266), Apr 2009, p.1073-1087.

\bibitem{FBH07}
F.B. Holt, \emph{Expected gaps between prime numbers}, arXiv:0706.0889, Jun 2007.

\bibitem{FBHSFU}
F.B. Holt \& H. Rudd, \emph{Combinatorics of the gaps between primes}, Connections in Discrete Mathematics, Simon Fraser U.,
arXiv 1510.00743, June 2015.

\bibitem{FBHPatterns}
F.B. Holt, \emph{Patterns among the Primes}, KDP, June 2022.

\bibitem{FBH2p}
F.B. Holt, \emph{Addendum: models for gaps $g=2p_1$}, arXiv:2309.16833v1, Sept 2023.

\bibitem{FBHktuple}
F.B. Holt, \emph{Eratosthenes sieve supports the $k$-tuple conjecture}, arXiv:2502.20470v3, July 2025.

\bibitem{HL}
G.H. Hardy and J.E. Littlewood, \emph{Some problems in 'partitio numerorum'
  iii: On the expression of a number as a sum of primes}, G.H. Hardy Collected
  Papers, vol.~1, Clarendon Press, 1966, pp.~561--630.
  
\bibitem{Sch76}
L. Shoenfeld, \emph{Sharper bounds for the Chebyshev functions $\theta(x)$ and $\psi(x)$}, Math. of Comp.,  30(134), Apr 1976, pp.~337-360.

\end{thebibliography}
\providecommand{\bysame}{\leavevmode\hbox to3em{\hrulefill}\thinspace}
\providecommand{\MR}{\relax\ifhmode\unskip\space\fi MR }
\providecommand{\MRhref}[2]{%
  \href{http://www.ams.org/mathscinet-getitem?mr=#1}{#2}
}
\providecommand{\href}[2]{#2}

\end{document}